\documentclass[12pt]{article}
\usepackage{graphicx}
\usepackage{epstopdf}
\usepackage{float}
\usepackage{color}
\usepackage{amssymb}
\usepackage{algorithm}
\usepackage{algorithmic}

\usepackage{mathtools}
\usepackage[thmmarks,amsmath]{ntheorem}

\usepackage{bm}

\usepackage{extarrows}
\usepackage[hang,flushmargin]{footmisc} 

\usepackage{mathrsfs} 
\usepackage[font=footnotesize,skip=0pt,textfont=rm,labelfont=rm]{caption,subcaption}

\usepackage{booktabs}
\usepackage{tocloft}
\usepackage{hyperref}
\hypersetup{hypertex=true,
	colorlinks=true,
	linkcolor=blue,
	anchorcolor=blue,
	citecolor=blue}
\usepackage{caption}
\usepackage{cite}

{
	\theoremstyle{nonumberplain}
	\theoremheaderfont{\bfseries}
	\theorembodyfont{\normalfont}
	\theoremsymbol{\mbox{$\Box$}}
	\newtheorem{proof}{Proof}
}


\usepackage{theorem}
\newtheorem{claim}{Claim}

\newtheorem{theorem}{Theorem}[section]

\newtheorem{lemma}{Lemma}[section]

\newtheorem{case}{Case}

\newcommand{\Rmnum}[1]{\uppercase\expandafter{\romannumeral #1}}

{
	\theoremheaderfont{\bfseries}
	\theorembodyfont{\rm}
	
}
\graphicspath{{figure/}}                                 
\usepackage[a4paper]{geometry}                           
\begin{document}
	\title{\textbf{Resolving Two Open Problems of Planar $B_k$-CPG Recognition: $k=0,1$}}
	
	\author{Bin Sun$^{1}$, Shou-Jun Xu$^{1}$, Yu Yang$^{2,\dagger}$\\
	{\rmfamily\small $^{1}$School of Mathematics and Statistics, Lanzhou University, Lanzhou, 730000, PR China }\\
	{\rmfamily\small $^{2}$School of Mathematics and Statistics, Huaibei Normal University, Huaibei, 235000, PR China}\\
}

\renewcommand{\thefootnote}{\fnsymbol{footnote}}

\footnotetext[2]{Corresponding author. E-mail address: yyang@chnu.edu.cn (Y. Yang).}
	
	\date{}

	\maketitle
    
	\textbf{Abstract:} A $k$-bend path is a non-self-intersecting polyline that lies on a grid and consists of at most $k+1$ axis-parallel line segments. A $B_k$-CPG graph is a graph whose vertices can be represented by pairwise interiorly disjoint $k$-bend paths on a grid such that two vertices are adjacent if and only if the corresponding grid paths touch at a grid point. We prove that recognizing planar $B_0$-CPG graphs of maximum degree 8 is NP-complete, and that recognizing planar $B_1$-CPG graphs of maximum degree 11 is NP-complete. These results settle two of the three planar recognition problems left open by Champseix, Galby, Munaro, and Ries.
	\vspace{1cm}	
	

\section{Introduction}
\par Intersection graphs of various types of objects have been extensively studied in the last sixty years. A $k$-bend path is a non-self-intersecting polyline that lies on a grid and consists of at most $k+1$ axis-parallel line segments. For 1-bend paths, there are six possible shapes: $\llcorner$, $\lrcorner$, $\ulcorner$, $\urcorner$, $\shortmid$, and $\text{-}$. The {\em Vertex intersection graph of Paths on a Grid} ({\em VPG graph}) was introduced in \cite{asinowski2012vertex}. In this representation, the vertices can be represented by paths on a grid such that two vertices are adjacent if and only if the corresponding paths intersect on at least one grid point. If each path on a grid has at most $k$ bends, then it is a {\em $B_{k}$-VPG graph}. 
\par A notion closely related to the VPG graph is that of the {\em CPG graph} which is a VPG graph such that the path interiors are pairwise disjoint, and two vertices are adjacent if and only if the corresponding paths touch at a grid point. Similarly to VPG graphs, a {\em $B_{k}$-CPG graph} is a CPG graph admitting a representation in which each path is a $k$-bend path. Clearly, every $B_{k}$-CPG graph is a $B_{k}$-VPG graph.

\par Golumbic et al. \cite{golumbic2009edge} introduced the {\em Edge intersection graph of Paths on a Grid} ({\em EPG graph}) whose vertices can be represented by paths on a grid such that two vertices are adjacent if and only if the corresponding paths share at least one grid edge.

\par VPG and EPG graphs were introduced to study electrical networks \cite{sinden1966topology} and have practical applications in areas such as circuit layout, where wires can be modeled as paths on a grid.

\par Gon{\c{c}}alves et al. \cite{gonccalves2018planar} proved a conjecture proposed in \cite{chaplick2012planar}: all triangle-free planar graphs are $B_{1}$-CPG with shape of $\llcorner$, $\shortmid$, or $\text{-}$. Chaplick et al. \cite{chaplick2012bend} proved that deciding whether a given $B_{k+1}$-VPG graph belongs to $B_{k}$-VPG is NP-complete. Deniz et al. \cite{deniz2018contact} showed that CPG graphs need not be planar and that not every planar graph is CPG. Within restricted CPG, its recognition problem can be more tractable. Bonomo-Braberman et al. \cite{bonomo2022some} obtained forbidden-induced-subgraph characterizations within several classes and a polynomial-time recognition algorithm for chordal graphs. Bonomo-Braberman et al. \cite{bonomo2020characterising} established a characterization and a recognition algorithm within circular-arc graphs. A number of algorithmic results concerning subclasses of VPG and EPG have also been obtained in \cite{silva2025clique,epstein2013approximation,wang2025algorithm,bose2022computing,galby2020approximating}. These results impose additional structural restrictions and do not determine the complexity on arbitrary planar input graphs.

\par Champseix et al. \cite{Champseix2021cpg} proved that recognizing $B_k$-CPG graphs is NP-complete for every fixed $k\geq1$. Their planar hardness result applies to every fixed $k\geq3$, and they also proved NP-completeness of recognizing planar CPG graphs without a prescribed bend bound. They explicitly left the planar cases $k=0,1,2$ open~\cite{Champseix2021cpg}. Their structural results further show that the bend hierarchy remains strict within planar graphs. Consequently, neither the inclusion $B_k\text{-CPG}\subseteq B_{k+1}\text{-CPG}$ nor hardness at a larger bend bound settles recognition at a smaller one.

\par The paper is organized as follows. Section 2 introduces terminology and preliminary results. Section 3 proves that recognizing planar $B_0$-CPG graphs is NP-complete. Section 4 proves that recognizing planar $B_1$-CPG graphs is NP-complete.

\section{Preliminaries}
\par In this paper, all graphs are simple, i.e., they contain no loops or multiple edges. Let $G=(V(G),E(G))$ be a graph. We denote by $N(v)$ the {\em open neighborhood} of a vertex $v\in V(G)$, and $N[v]=N(v)\cup\{v\}$ is the {\em closed neighborhood} of $v$. The {\em degree} of $v$ is denoted by $d(v)=|N(v)|$. Let $\Delta(G)$ and $\delta(G)$ denote {\em the maximum degree} and {\em minimum degree} of $G$, respectively. A graph $G'$ is an {\em induced subgraph} of $G$ if $V(G')\subseteq V(G)$ and $E(G')=\{uv\mid uv\in E(G)\text{ and }u,v\in V(G')\}$. A graph $G$ is a {\em subcubic graph} if $\Delta(G)\leq 3$.

\par We denote by $P_v$ the path representing a vertex $v$. For a path $P_v$, its endpoint set is $\partial{P_v}$, its interior is $P_v^{o}=P_v\setminus\partial{P_v}$, and its number of bends is $b(P_v)$. In particular, a bend belongs to the interior of the path. A {\em segment} is a maximal straight portion of a grid path. A point is a {\em free endpoint} of a representation if it is the endpoint of one path and belongs to no other path. Freeness is always relative to the representation under consideration; deleting paths can create free endpoints. Restricting a representation to a vertex set gives a representation of the induced subgraph. At a contact point, distinct paths cannot use the same incident grid direction. A path passing through the point uses two directions, whether or not it bends there. When we speak of an open ray from a point $p$, the ray does not contain $p$.
\par In every figure depicting a $B_k$-CPG representation, each path is drawn with an arrowhead at each of its two endpoints.

\begin{lemma}\label{le2.1}
For every fixed $k\geq0$, $B_k$-CPG recognition belongs to NP. Every representable $n$-vertex graph has a certificate with coordinates in $\{1,\ldots,(k+2)n\}$ and total bit length $O(n\log n)$.
\end{lemma}

\begin{proof}
There are at most $(k+2)n$ endpoints and bends. Replace each distinct horizontal coordinate by its rank among all horizontal coordinates, and do the same for the vertical coordinates. Strict order and equality in each coordinate are preserved. Hence segment intersections, endpoint incidences, overlaps, and non-intersections are preserved, as are the simplicity of each path and its bend count. Each coordinate now uses $O(\log n)$ bits. A verifier checks that the segments are axis-parallel, that each path is nontrivial and simple, that it has at most $k$ bends, that interiors of different paths are disjoint, and that the contact graph is the given graph. For fixed $k$, pairwise segment tests require polynomial time.
\end{proof}

\par An {\em exposed interior point} of a represented path is a point in the interior of a straight segment which belongs to no other path and from which some perpendicular open ray avoids all represented paths. We also say that a free endpoint is {\em exposed} if the open ray continuing its terminal segment outwards avoids all paths.

\begin{lemma}\label{le2.2}
\par Fix $k\geq0$, a graph $D$, and a set $S\subseteq V(D)$. Let $J$ be a $B_k$-CPG graph with a fixed vertex $s$ such that, in every $B_k$-CPG representation of $J$, $P_s$ has no free endpoint. 
\par Suppose $J$ has a $B_k$-CPG representation with an exposed interior point on $P_s$. Attach a copy of $J$ to each $t\in S$ by the edge $ts$. The resulting graph $G$ is $B_k$-CPG if and only if $D$ has a $B_k$-CPG representation in which each $P_t$, $t\in S$, has a free endpoint.
\par If $D$ and $J$ are planar, the above operation preserves planarity.
\end{lemma}

\begin{proof}
\par Consider a $B_k$-CPG representation of $G$. An endpoint of $P_s$ belongs to a path $P_x$ for some $x\neq s$ within the copy of $J$. Thus $P_t$ cannot contain that endpoint: it would create the forbidden edge $tx$. The edge $ts$ must use an endpoint of $P_t$ in the interior of $P_s$. Since $s$ is adjacent to no vertex of $V(D)\setminus\{t\}$, this endpoint cannot belong to a path of $D$ other than $P_t$. Deleting all copies of $J$ therefore leaves a free endpoint on each path of $S$.

\par Conversely, start with a $B_k$-CPG representation of $D$ having the stated free endpoints. Around each chosen endpoint, take a sufficiently small neighborhood which meets only the terminal segment of its path; choose these neighborhoods pairwise disjoint. Scale and rotate the fixed representation of $J$ into the neighborhood so that its exposed interior point is the chosen endpoint and its empty ray contains the incoming terminal segment. The new representation creates exactly the edge $ts$. No existing path changes its number of bends.
\par Each copy is joined to $D$ by one edge. Standard planar gluing operations apply: draw the copy in a small disk incident with the attachment vertex, taking that vertex on the outer face of the copy.
\end{proof}

\section{Zero-bend CPG recognition}
\par In this section, we show that \textsc{Recognition} is NP-complete for planar $B_0$-CPG graphs, via a reduction from the \textsc{HV-Rectilinear Planarity} problem.

\begin{lemma}\label{le3.1}
In a $B_0$-CPG representation, the segments of a triangle have a unique common point. The segments of a $K_4$ also have a unique common point, and this point is an endpoint of all four segments. In particular, the diamond $K_4-e$ is not $B_0$-CPG.
\end{lemma}
\begin{proof}
\par Let $T$ be a triangle whose three edges are represented by axis-parallel segments $P_{v_1}, P_{v_2}, P_{v_3}$ in a $B_0$-CPG representation. Each segment is the Cartesian product of two closed intervals, at least one of which is a singleton. Since the edges of $T$ are pairwise adjacent, the segments $P_{v_1}, P_{v_2}, P_{v_3}$ are pairwise intersecting, so their horizontal coordinate intervals are pairwise intersecting and likewise for their vertical coordinate intervals. For closed intervals, pairwise intersection implies a common point: the largest left endpoint does not exceed the smallest right endpoint. Applying this fact in both coordinates yields a point $p$ common to all three segments. This point is unique: two distinct segments sharing two points would share a non-degenerate common subsegment, contrary to interior disjointness. Hence the segments of $T$ share a unique common point.
\par The same argument shows that the four segments of a $K_4$ share a unique common point. If one segment passed through this point, it would use two grid directions. The other three segments would each need a different unused direction, but only two would remain. Hence all four segments end at the point and occupy its four directions.

\par The two triangles of a diamond share an edge. Their common points must coincide, by uniqueness of the intersection of the two segments representing that edge. All four segments would then meet at one point, creating the missing edge. This proves the last assertion.
\end{proof}

\par To {\em attach a private $K_4$ at $v_1$}, add three vertices which, together with $v_1$, induce a $K_4$, and give the three new vertices no other neighbors. By Lemma \ref{le3.1}, this operation occupies one endpoint of $P_{v_1}$. Two private copies occupy different endpoints: if their common points coincided, vertices in different copies would become adjacent. A segment not adjacent to the three private vertices cannot contain the occupied endpoint.

\begin{lemma}\label{le3.2}
Let $v_1,v_2,v_3$ induce a triangle, and attach two private copies of $K_4$ at $v_1$. In every $B_0$-CPG representation of the resulting graph, $P_{v_2}$ and $P_{v_3}$ are collinear and share a common endpoint, while $P_{v_1}$ passes perpendicularly through that point.
\end{lemma}
\begin{proof}
The triangle has a common point $r$ by Lemma \ref{le3.1}. Both endpoints of $P_{v_1}$ are occupied by its private copies. Since $v_2,v_3$ have no neighbors in private copies, $r$ is not an endpoint of $P_{v_1}$. Thus $P_{v_1}$ uses two opposite directions at $r$. The segments $P_{v_2},P_{v_3}$ must end there and use the other two opposite directions. They are therefore collinear and perpendicular to $P_{v_1}$.
\end{proof}

\par An instance of \textsc{HV-Rectilinear Planarity} is a planar graph $G$ with an edge-label function $\lambda:E(G)\to\{H,V\}$. It asks for a planar straight-line drawing in which every $H$-edge is horizontal and every $V$-edge is vertical. Vertices must be distinct, and edges may meet only at common endvertices. The embedding is not fixed. Didimo et al. \cite[Theorem~1]{didimo2014complexity} proved that this problem is NP-complete for subcubic graphs.

\begin{figure}
\centering
  \setlength{\unitlength}{1bp}%
   \begin{picture}(418.54, 136.76)(0,0)
  \put(0,0){\includegraphics{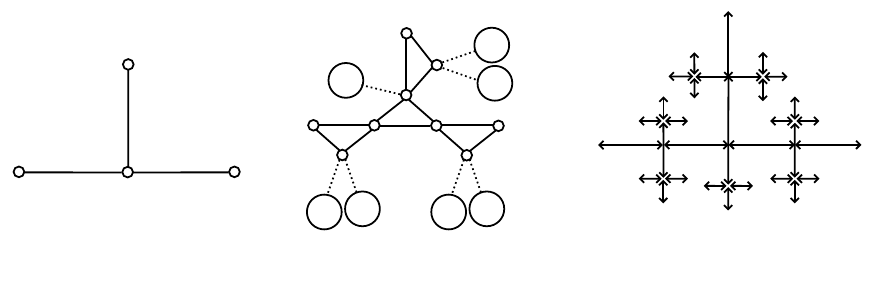}}
  \put(57.87,44.27){\fontsize{9.72}{11.66}\selectfont $u_1$}
  \put(5.67,44.39){\fontsize{9.72}{11.66}\selectfont $u_2$}
  \put(109.12,44.39){\fontsize{9.72}{11.66}\selectfont $u_3$}
  \put(64.59,103.39){\rotatebox{0.00}{\fontsize{9.72}{11.66}\selectfont \smash{\makebox[0pt][l]{$u_4$}}}}
  \put(26.01,57.12){\fontsize{9.72}{11.66}\selectfont $e_1$}
  \put(80.28,57.32){\fontsize{9.72}{11.66}\selectfont $e_2$}
  \put(62.50,79.73){\fontsize{9.72}{11.66}\selectfont $e_3$}
  \put(228.37,112.10){\fontsize{9.72}{11.66}\selectfont $K_3$}
  \put(229.88,93.79){\fontsize{9.72}{11.66}\selectfont $K_3$}
  \put(148.66,32.19){\rotatebox{0.00}{\fontsize{9.72}{11.66}\selectfont \smash{\makebox[0pt][l]{$K_3$}}}}
  \put(167.10,32.89){\rotatebox{0.00}{\fontsize{9.72}{11.66}\selectfont \smash{\makebox[0pt][l]{$K_3$}}}}
  \put(208.32,32.61){\rotatebox{0.00}{\fontsize{9.72}{11.66}\selectfont \smash{\makebox[0pt][l]{$K_3$}}}}
  \put(225.65,32.89){\rotatebox{0.00}{\fontsize{9.72}{11.66}\selectfont \smash{\makebox[0pt][l]{$K_3$}}}}
  \put(158.35,95.17){\fontsize{9.72}{11.66}\selectfont $K_3$}
  \put(57.87,7.76){\fontsize{9.72}{11.66}\selectfont (a)}
  \put(189.31,7.76){\fontsize{9.72}{11.66}\selectfont (b)}
  \put(342.67,7.76){\fontsize{9.72}{11.66}\selectfont (c)}
  \end{picture}%
  \caption{(a) An HV drawing of $K_{1,3}$. (b) The graph $X(K_{1,3})$. The dotted edges indicate a vertex that, together with a $K_3$, forms a $K_4$. (c) The $B_{0}$-CPG representation of $X(K_{1,3})$; restricting attention to the four triple contact points and the routes between them gives exactly an HV drawing of $K_{1,3}$.}
  \label{fig1}
\end{figure}

\par We first normalize an instance. Delete isolated vertices. If a degree-three vertex has three incident edges of the same label, the instance is negative: a horizontal or vertical line supplies only two directions at the vertex. In this case output the diamond which is a fixed planar negative instance of $B_0$-CPG recognition.

\par For each degree-two vertex, add one leaf. Give the new edge label $V$ when both old edges have label $H$, label $H$ when both have label $V$, and label $H$ when the old labels differ. This preserves the answer. Deleting the new edges proves one implication. For the other, take a sufficiently small neighborhood of the old vertex in an HV drawing. Two edges with the same label already occupy the two opposite directions of that axis, leaving both perpendicular directions unused. Two edges with different labels leave the opposite horizontal direction unused. We can place the new leaf on an arbitrarily short segment in an appropriate unused direction, without meeting any other part of the drawing.

\par After this normalization, every vertex has degree one or three, and the three labels at a degree-three vertex are either $H,H,V$ or $V,V,H$. Let $e^*(u)$ be the unique edge whose label differs from the other two at such a vertex $u$. The normalization runs in polynomial time and preserves planarity and the maximum-degree bound. Below, $G$ denotes the normalized nonempty instance.

\par Construct an unlabeled graph $X(G)$ as follows. For every incidence $u\in e$, introduce an incidence vertex $\ell_{u,e}$. Furthermore, if $d_{G}(u)\in\{1,3\}$, then $\ell_{u,e}$ is called an incidence vertex of degree one or three. For every edge $e=uv$, introduce $z_e$, make $\ell_{u,e},\ell_{v,e},z_e$ a triangle, and attach two private copies of $K_4$ at $z_e$. For every degree-three vertex $u$, make its three incidence vertices a triangle and attach one private $K_4$ at $\ell_{u,e^*(u)}$. Add no other edges. Fig. \ref{fig1}(b) gives an example of $X(K_{1,3})$ constructed from $K_{1,3}$.

\par Let $m=|E(G)|$ and let $n_3$ be the number of degree-three vertices of $G$. There are $2m$ incidence vertices, $m$ vertices $z_e$, $6m$ private vertices at the edge gadgets, and $3n_3$ private vertices at the vertex gadgets. Each edge gadget contributes three triangle edges and six edges per private-$K_4$ (three triangle edges plus three edges to $z_e$), for a total of twelve private edges.  Each degree-three vertex gadget contributes three triangle edges and six edges per private-$K_4$. Hence
\begin{equation}\label{eq1}
|V(X(G))|=9m+3n_3\text{ and } |E(X(G))|=15m+9n_3.
\end{equation}

\par The construction preserves planarity independently of whether the HV instance is positive. Start from any planar embedding of $G$. Replace each degree-three vertex by a small triangle of incidence vertices, respecting their cyclic order, and route each edge between its two incidence vertices along its old embedding. Add $z_e$ and its two incident edges in a narrow region on one side of the edge $\ell_{u,e}\ell_{v,e}$. Finally, each private $K_4$ meets the graph built so far in only one vertex, so we can add it in a small incident disk.

\par A vertex $z_e$ has degree $2+6=8$. A degree-one incidence vertex has degree two. A degree-three incidence vertex has degree four before a possible private attachment and degree seven afterwards. Every private vertex has degree three. Thus
\begin{equation}\label{eq2}
\Delta(X(G))\leq8.
\end{equation}

\begin{lemma}\label{le3.3}
The normalized instance $G$ has an HV drawing if and only if $X(G)$ has a $B_0$-CPG representation.
\end{lemma}
\begin{proof}
Suppose first that $G$ has an HV drawing. For every edge $e=uv$, choose an interior point $q_e$ and represent $\ell_{u,e},\ell_{v,e}$ by the two half-edge segments from $u,v$ to $q_e$.

\par At a degree-three vertex $u$, the two edges with the same label leave $u$ in opposite directions, and $e^*(u)$ leaves in a perpendicular direction. Extend the segment representing $\ell_{u,e^*(u)}$ a short distance through $u$ into the fourth, unused direction. The other two segments continue to end at $u$. The three segments now represent the vertex triangle, with the specified segment passing through the common point. At its new endpoint, place three still shorter segments in the other three directions to realize the private $K_4$.

\par At every $q_e$, place a short perpendicular segment $P_{z_e}$ through the point. The two half-edge segments ($\ell_{u,e},\ell_{v,e}$ ) end at $q_e$, so these three segments realize the edge triangle. At each endpoint of $P_{z_e}$, add three shorter segments to represent its private $K_4$.

\par Choose pairwise disjoint neighborhoods of the vertices of $V(G)$ and the selected points $q_e$. All modifications fit inside these neighborhoods. Within each neighborhood, choose the private-$K_4$ segments shorter than the distance from their common point to every other segment not intended to meet them. No unintended contact or overlap results. All paths remain straight. Original real coordinates pose no difficulty: replacing the vertex coordinates by ranks first yields an integer HV drawing with the same intersections and non-intersections; the local modifications may then use rational lengths, followed by a common scaling.

\begin{figure}
\centering
  \setlength{\unitlength}{1bp}%
 \begin{picture}(415.07, 105.24)(0,0)
  \put(0,0){\includegraphics{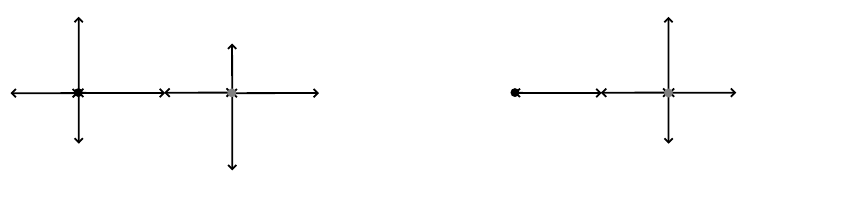}}
  \put(73.21,7.76){\fontsize{9.72}{11.66}\selectfont (a)}
  \put(113.09,31.35){\fontsize{9.72}{11.66}\selectfont $P_{\ell_{v,e^*(v)}}$}
  \put(39.55,91.98){\fontsize{9.72}{11.66}\selectfont $P_{\ell_{u,e^*(u)}}$}
  \put(47.22,67.08){\fontsize{9.72}{11.66}\selectfont $P_{\ell_{u,e}}$}
  \put(82.13,67.08){\fontsize{9.72}{11.66}\selectfont $P_{\ell_{v,e}}$}
  \put(267.25,7.76){\fontsize{9.72}{11.66}\selectfont (b)}
  \put(321.47,85.22){\fontsize{9.72}{11.66}\selectfont $P_{\ell_{v,e^*(v)}}$}
  \put(256.76,67.08){\fontsize{9.72}{11.66}\selectfont $P_{\ell_{u,e}}$}
  \put(291.66,67.08){\fontsize{9.72}{11.66}\selectfont $P_{\ell_{v,e}}$}
  \end{picture}%
  \caption{The black and gray points represent $c_u$ and $c_v$, respectively. $P_{\ell_{u,e}}\cup P_{\ell_{v,e}}$ is the horizontal segment connecting $u$ and $v$ in the HV drawing of $G$, where $e=uv\in E(G)$. (a) $u$ and $v$ have degree three in $G$; (b) $u$ and $v$ have degrees one and three, respectively, in $G$.}
  \label{fig7}
\end{figure}

\par Conversely, fix any $B_0$-CPG representation of $X(G)$. For every $e=uv\in E(G)$, Lemma \ref{le3.2} gives a point $q_e$ which is an endpoint of both $P_{\ell_{u,e}}$ and $P_{\ell_{v,e}}$. The two segments are collinear and extend in opposite directions from $q_e$. The segment $P_{z_e}$ passes through $q_e$ perpendicularly.

\par For a degree-three vertex $u$, Lemma \ref{le3.1} gives a common point $c_u$ of its three incidence segments (the black point in Fig. \ref{fig7}(a) is $c_u$). Consider $S^*=P_{\ell_{u,e^*(u)}}$. One endpoint of $S^*$ is $q_{e^*(u)}$. Its private $K_4$ occupies an endpoint $a_u$. $a_u\neq q_{e^*(u)}$: if they coincided, the private vertices would also touch $P_{z_{e^*(u)}}$ and the opposite incidence segment, creating forbidden edges. Thus they are exactly the two endpoints of $S^*$.

\par The point $c_u$ is neither of them. If $c_u=q_{e^*(u)}$, the other two incidence segments at $u$ would meet $P_{z_{e^*(u)}}$, although their vertices are not adjacent to $z_{e^*(u)}$. If $c_u=a_u$, those segments would meet the three private segments, again producing forbidden edges. Therefore
\begin{equation}\label{eq3}
c_u\in P_{\ell_{u,e^*(u)}}^o.
\end{equation}
The other two incidence segments must end at $c_u$, use the two opposite perpendicular directions, and hence be collinear. This establishes the prescribed equality and inequality relations among the three edge labels at $u$. Also, $c_u \neq q_e$ for every edge $e$ incident with $u$. If $c_u = q_e$, then the two other segments incident with $u$ would meet the non-neighbor $z_e$.

\par For a degree-one vertex $u$ incident with $e$, let $c_u$ be the endpoint of $P_{\ell_{u,e}}$ other than $q_e$ (the black point in Fig. \ref{fig7}(b) is $c_u$). It is free. Indeed, the incidence vertex has only two neighbors, and its contacts with both of them already occur at $q_e$. Since in a $B_0$-CPG representation every path is a straight segment, two segments meet in at most one point. Hence no path other than $P_{z_e}$ and $P_{\ell_{v,e}}$ can meet $P_{\ell_{u,e}}$ a second time, and $c_u$ belongs to no other path.

\par Place every vertex $u\in V(G)$ at $c_u$. For $e=uv$, draw the segment
\begin{equation}\label{eq4}
\Gamma(e)=[c_u,c_v]\subseteq P_{\ell_{u,e}}\cup P_{\ell_{v,e}}.
\end{equation}
The two original incidence segments extend in opposite directions from $q_e$, and neither $c_u$ nor $c_v$ equals $q_e$. Thus this is a non-degenerate horizontal or vertical segment. Fig. \ref{fig7}(a) and (b) show the positions of $c_u$ (black point) and $c_v$ (gray point) in the $B_0$-CPG representation of $X(G)$ for the cases $d_G(u)=d_G(v)=3$ and $d_G(u)=1,\ d_G(v)=3$, respectively.

\par We verify all conditions for a planar drawing. Adjacent vertices of $G$ are distinct; their incident segments meet only at $q_e$, while their chosen points differ from $q_e$. For two nonadjacent vertices of $G$, choose an incidence segment at each; the corresponding incidence vertices are nonadjacent in $X(G)$, so the segments are disjoint, and their chosen points cannot coincide.

\par For two edges of $G$ with no common endvertex, every incidence vertex of one is nonadjacent to every incidence vertex of the other. Their four representing segments are consequently disjoint across the two edges, and \eqref{eq4} excludes any intersection of the recovered edges. If two edges share $u$, the only adjacent pair among their four incidence vertices is the pair at $u$. Its two segments intersect exactly at $c_u$, so the recovered edges meet only at that point. A vertex not incident with an edge cannot lie on it either: an incidence segment containing the vertex point would then meet a nonadjacent incidence segment used by that edge. Hence 
\[
\Gamma=\bigcup_{e\in E(G)}\Gamma(e)
\]
is a planar straight-line drawing of $G$.

\par It remains to recover the absolute labels, rather than only the local parallelism relations. Encode horizontal and vertical directions by 0 and 1. Let $\theta(e)$ be the direction in $\Gamma$ and let $\lambda(e)$ be the input label in the same encoding. At a degree-three vertex, \eqref{eq3} implies, for any two incident edges $e,f$, that
\begin{equation}\label{eq5}
\theta(e)\bmod 2+\lambda(e)\bmod 2=\theta(f)\bmod 2+\lambda(f)\bmod 2.
\end{equation}
Indeed, the two edges with equal input labels are represented by collinear segments, and the third is perpendicular to them. Since every vertex has degree one or three, \eqref{eq5} propagates along each connected component containing an edge. The value $\theta(e)\bmod 2+\lambda(e)\bmod 2$ is constant on that component. If it is one, rotate the entire component through $90^\circ$; if it is zero, leave the component unchanged. Translating different components into disjoint regions yields an HV drawing. Deleting the leaves introduced during normalization recovers a drawing of the original instance.
\end{proof}

\begin{theorem}\label{th3.1}
Recognizing planar $B_0$-CPG graphs is NP-complete even for graphs of maximum degree at most 8.
\end{theorem}

\begin{proof}
The normalization and construction of $X(G)$ runs in polynomial time. Equations \eqref{eq1} and \eqref{eq2} give linear output size and the asserted degree bound, and the output is always planar. Lemma \ref{le3.3} proves the reduction from the subcubic NP-complete problem in~\cite{didimo2014complexity}. Lemma \ref{le2.1} proves membership in NP.
\end{proof}

\section{One-bend CPG recognition}
\par In this section, we show that \textsc{Recognition} is NP-complete for planar $B_1$-CPG graphs, via a reduction from the \textsc{Planar Exactly 3-Bounded 3-Sat} problem.

\subsection{The end-eating gadget}
\begin{figure}
\centering
  \setlength{\unitlength}{1bp}%
 \begin{picture}(385.18, 161.40)(0,0)
  \put(0,0){\includegraphics{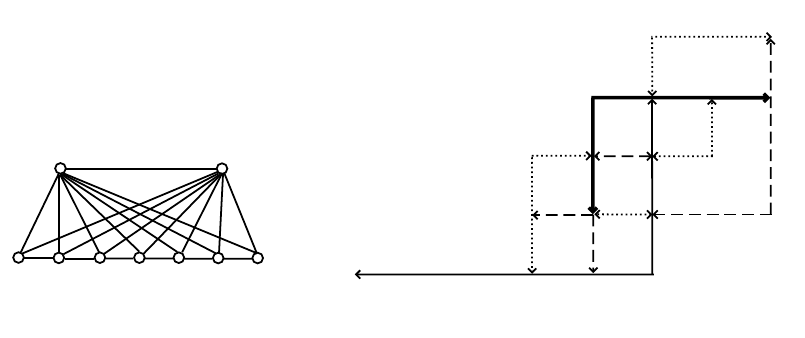}}
  \put(5.67,28.78){\fontsize{9.72}{11.66}\selectfont $u_1$}
  \put(24.61,28.78){\fontsize{9.72}{11.66}\selectfont $u_2$}
  \put(45.92,28.78){\fontsize{9.72}{11.66}\selectfont $u_3$}
  \put(64.86,28.78){\fontsize{9.72}{11.66}\selectfont $u_4$}
  \put(83.80,28.78){\fontsize{9.72}{11.66}\selectfont $u_5$}
  \put(23.63,86.99){\fontsize{9.72}{11.66}\selectfont $v$}
  \put(101.18,86.60){\fontsize{9.72}{11.66}\selectfont $w$}
  \put(102.71,28.61){\fontsize{9.72}{11.66}\selectfont $u_6$}
  \put(121.65,28.61){\fontsize{9.72}{11.66}\selectfont $u_7$}
  \put(50.60,7.76){\fontsize{9.72}{11.66}\selectfont (a) $H_1$}
  \put(268.18,108.05){\fontsize{9.72}{11.66}\selectfont $P_{w}$}
  \put(324.33,148.14){\fontsize{9.72}{11.66}\selectfont $P_{u_7}$}
  \put(223.76,20.88){\fontsize{9.72}{11.66}\selectfont $P_{v}$}
  \put(323.95,96.20){\fontsize{9.72}{11.66}\selectfont $P_{u_1}$}
  \put(335.79,63.62){\fontsize{9.72}{11.66}\selectfont $P_{u_6}$}
  \put(291.79,91.44){\fontsize{9.72}{11.66}\selectfont $P_{u_2}$}
  \put(291.79,64.89){\fontsize{9.72}{11.66}\selectfont $P_{u_5}$}
  \put(285.87,41.19){\fontsize{9.72}{11.66}\selectfont $P_{u_4}$}
  \put(239.04,70.85){\fontsize{9.72}{11.66}\selectfont $P_{u_3}$}
  \put(301.90,7.76){\fontsize{9.72}{11.66}\selectfont (b)}
  \end{picture}%
\caption{The graph $H_1$ and its $B_1$-CPG representation.}
\label{fig3}
\end{figure}

\begin{figure}
\centering
  \setlength{\unitlength}{1bp}%
   \begin{picture}(275.57, 258.34)(0,0)
  \put(0,0){\includegraphics{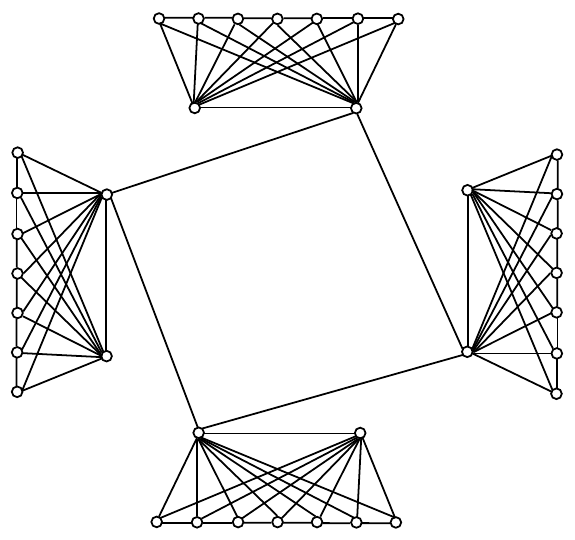}}
  \put(206.35,90.85){\fontsize{9.72}{11.66}\selectfont $v_0$}
  \end{picture}%
  \caption{The graph $F$, in which $v_0$ is the only vertex that can be adjacent to a vertex outside $V(F)$.}
  \label{fig8}
\end{figure}

\par To construct a planar end-eating gadget $F$, we first introduce the graph $H_1$. Let $H_1=K_2\vee P_7$ (depicted in Fig. \ref{fig3}(a)), with universal vertices $v,w$ and path $u_1u_2\dots u_7$. We refer to $P_v$ and $P_w$ as the {\em hosts} of $P_{u_i}$. The graph $H_1$ is planar.

\par The end-eating graph $F$ is constructed by taking four vertex-disjoint copies of $H_1$, denoted $H_1^0,\dots,H_1^3$ with roots $v_0,\dots,v_3$, and adding the edges $v_iv_{i+1}$ for indices modulo four (see Fig. \ref{fig8}). We use this graph in the proof of the NP-hardness of \textsc{Recognition} restricted to planar $B_1$-CPG graphs.

\par We aim to prove that in every $B_1$-CPG representation of $F$, the path $P_{v_0}$ has no free endpoint. To this end, we first prove that in every $B_1$-CPG representation of $H_1$, the path $P_v$ has at most one free endpoint.

\par A 1-bend path is monotone in both coordinates and lies in the axis-parallel rectangle determined by its endpoints. We will use this elementary property throughout the next proof. In particular, if the path leaves an endpoint in the downward direction, it never returns to that endpoint's row.

\par For two 1-bend paths in which share a common point $p$, a path that starts from $p$ and goes to the left is called a {\em left-going path}, a path that starts from $p$ and goes to the right is called a {\em right-going path}, a path that starts at $p$ and goes downward is called a {\em downward path}; and a path that starts at $p$ and goes upward is called an {\em upward path}.

\begin{lemma}\label{le4.1}
In every $B_1$-CPG representation of $H_1$, each of $P_v$ and $P_w$ has at most one free endpoint.
\end{lemma}
\begin{proof}
The two universal vertices are interchangeable, so it suffices to consider $P=P_v$. Suppose, for a contradiction, that both endpoints of $P$ are free, and put $Q=P_w$.

\par We may assume that $P$ has a bend; the case analysis below requires a designated bend point splitting $P$ into two arms. If $P$ is straight, proceed as follows. Refine the grid so that a small open neighborhood of a free endpoint of $P$, one meeting no point of any other path, contains new grid points. Shorten the terminal portion of $P$ slightly inside this neighborhood and append a short perpendicular segment terminating at a new grid point of the neighborhood. The resulting path has exactly one bend, so the representation remains a $B_1$-CPG representation. The modification is confined to a region containing no point of any other path; hence every contact of every path, in particular the contact realizing the edge $vw$, is unchanged. The new endpoint lies in the same region and is still free, and the opposite endpoint of $P$ is untouched, so both endpoints of $P$ remain free.

\par After a translation, rotation, and reflection, write
\[
P=(-a,0)\to(0,0)\to(0,h),\qquad a,h>0.
\]
Since the endpoints of $P$ are free, the edge $vw$ uses an endpoint of $Q$ in the interior of $P$. Interchanging the two arms of $P$ if necessary, choose such an endpoint $q=(c,0)$ with $-a<c\leq0$. Write
\[
Q=(c,0)\to(c,b)\to(d,b),\qquad b\ne0.
\]
The case $d=c$ means that $Q$ is straight. If $c=0$, the two unused directions at the bend of $P$ are right and down; interchanging its arms allows us to assume that $Q$ starts downwards, so $b<0$. Let $r=(d,b)$ be the other endpoint of $Q$.

\par Let $B$ be the set of vertices among $u_1,\dots,u_7$ whose paths contain $r$. These vertices are pairwise adjacent in the induced path $P_7$. Hence
\begin{equation}\label{eq6}
|B|\leq2,\text{ and the two vertices of $B$ are adjacent if }|B|=2.
\end{equation}
At most one of the seven paths contains $q$: the interior of $P$ uses two directions at $q$, $Q$ uses a third, and only one remains. If such a path exists, denote it by $Z$. The point $q$ is an endpoint of $Z$.

\par Call a path among $P_{u_1},\dots,P_{u_7}$ {\em ordinary} if it contains neither $q$ nor $r$. In the following, we prove further properties of ordinary paths.

\begin{claim}\label{claim1}
Every ordinary path has exactly one endpoint in the interior of $P$, its other endpoint in the interior of $Q$, and no further point on $P\cup Q$. Consequently, any contact between two ordinary paths is a shared endpoint lying on $P$ or on $Q$.   
\end{claim} 

\par Let $R$ be an ordinary path, representing a vertex $u_i$. Since $u_i$ is adjacent to both $v$ and $w$, $R$ touches $P$ and touches $Q$. As both endpoints of $P$ are free, the contact with $P$ uses an endpoint of $R$ placed in the interior of $P$. As $R$ avoids $q$ and $r$, the contact with $Q$ cannot use an endpoint of $Q$; interior-disjointness then forces the contact with $Q$ to use the other endpoint of $R$, placed in the interior of $Q$.
\par Since the endpoints of $P$ are free, $P\cap Q\subseteq\{q,r\}$. The two contacts of $R$ are therefore distinct, $R$ has no further point on the hosts, and any contact between two ordinary paths is a shared endpoint on $P$ or on $Q$. $\square$

\begin{figure}
\centering
  \setlength{\unitlength}{1bp}%
  \begin{picture}(350.35, 118.97)(0,0)
  \put(0,0){\includegraphics{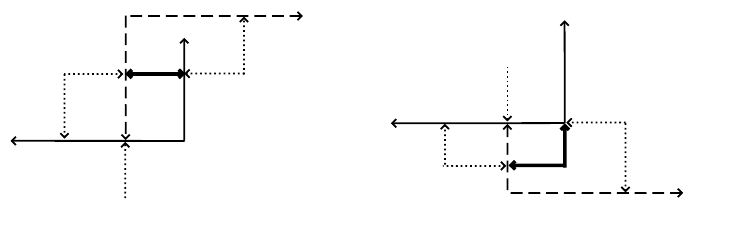}}
  \put(18.58,70.84){\fontsize{9.72}{11.66}\selectfont $A$}
  \put(65.92,87.18){\fontsize{9.72}{11.66}\selectfont $M$}
  \put(118.79,87.18){\fontsize{9.72}{11.66}\selectfont $C$}
  \put(60.87,102.80){\fontsize{9.72}{11.66}\selectfont $Q$}
  \put(88.75,53.06){\fontsize{9.72}{11.66}\selectfont $P$}
  \put(328.01,23.35){\fontsize{9.72}{11.66}\selectfont $Q$}
  \put(68.36,10.50){\fontsize{9.72}{11.66}\selectfont (a)}
  \put(258.67,7.76){\fontsize{9.72}{11.66}\selectfont (b)}
  \put(61.47,31.13){\fontsize{9.72}{11.66}\selectfont $Z$}
  \put(200.59,39.31){\fontsize{9.72}{11.66}\selectfont $A$}
  \put(254.26,42.62){\fontsize{9.72}{11.66}\selectfont $M$}
  \put(301.04,45.43){\fontsize{9.72}{11.66}\selectfont $C$}
  \put(259.91,94.92){\fontsize{9.72}{11.66}\selectfont $P$}
  \put(244.46,72.76){\fontsize{9.72}{11.66}\selectfont $Z$}
  \end{picture}%
  \caption{Two ways of placing three consecutive ordinary paths.}
  \label{fig6}
\end{figure}

\begin{claim}\label{claim2}
 When $c<0$, the contact graph of the ordinary paths consists only of isolated vertices and independent edges.   
\end{claim}
\par First, two ordinary paths cannot share an endpoint in the interior of the horizontal segment of $P$. Their endpoints on $Q$ lie strictly on the side of $y=0$ containing $Q\setminus\{q\}$. Coordinate monotonicity forces both paths to leave such a common endpoint into that same side, using the same direction and overlapping.

\par Suppose now that ordinary paths $A,M,C$ represent three consecutive vertices of the induced path $P_7$. The contacts of $M$ with $A,C$ occur at different endpoints of $M$. Otherwise, three ordinary paths would end at one host point, but the host interior already uses two directions. By interchanging $A,C$, assume that $A$ and $M$ share their endpoint on $Q$, while $M$ and $C$ share their endpoint on $P$.

\begin{case}\label{case1}
The contact point between $M$ and $C$ is $(0,t)$ with $0<t<h$, and the two paths leave it to the left and right.
\end{case}

\par The right-going path must end on the horizontal segment of $Q$ after one vertical turn. Consequently, $d>0$, and this horizontal segment crosses the line $x=0$. The range $0<b\leq h$ is impossible: it would meet the interior of $P$ or its free upper endpoint.

\par The alternative $b<0$ is also impossible. In that case the left-going path cannot end on the vertical segment of $Q$, which lies below $y=0$. To end on its horizontal segment, the path would have to turn vertically at a coordinate $j$ with $c<j<0$ and cross the horizontal segment of $P$ at $(j,0)$. Turning at $j=c$ would instead overlap the vertical segment of $Q$, and a turn at $j<c$ cannot reach its horizontal segment. These are all possibilities with at most one bend. Thus $b>h$.

\par All endpoints on $P$ now lie below the horizontal segment of $Q$. No two ordinary paths can share an endpoint in the interior of the horizontal segment of $Q$, since both must enter it from below. Nor can two ordinary paths share the bend $(c,b)$ of $Q$: one of the two unused directions there points upwards, whereas every endpoint on $P$ has height at most $h<b$. Therefore the right-going path cannot be the middle path $M$. It is $C$, and $M$ must run straight leftwards to the vertical segment of $Q$. The path $A$ enters the same point from the left and must start on the horizontal segment of $P$ further to the left. We obtain the forced form 
\begin{equation}\label{eq7}
A=(x,0)\to(x,t)\to(c,t),M=(c,t)\to(0,t),C=(0,t)\to(\alpha,t)\to(\alpha,b),
\end{equation}
with $-a<x<c<0<\alpha<d\text{ and }0<t<h<b$ (see Fig. \ref{fig6}(a)). The inequalities exclude free host endpoints, the two endpoints $q$ and $r$ of $Q$, and overlaps with the host segments.

\begin{case}\label{case2}
The contact point between $M$ and $C$ is the bend $(0,0)$ of $P$, the two paths leave it downwards and to the right.
\end{case}

\par A downward path cannot reach $Q$ with at most one bend unless $b<0$, and the right-going path again requires $d>0$. The right-going path ends on the horizontal segment of $Q$, which lies below every endpoint on $P$. It cannot share that endpoint with another ordinary path. It is therefore $C$, while $M$ goes downwards and then leftwards to the vertical segment of $Q$. The path $A$ approaches that contact from the left. Thus
\begin{equation}\label{eq8}
A=(x,0)\to(x,t)\to(c,t),M=(c,t)\to(0,t)\to(0,0),C=(0,0)\to(\alpha,0)\to(\alpha,b),
\end{equation}
with $-a<x<c<0<\alpha<d\text{ and }b<t<0$ (see Fig. \ref{fig6}(b)). There is no other location for a common endpoint on $P$: its horizontal interior was excluded, and its endpoints are free.

\par By \textbf{Cases} \ref{case1} and \ref{case2}, we show that in either \eqref{eq7} or \eqref{eq8}, these three ordinary paths form a whole connected component of the contact graph on $P_{u_1},\dots,P_{u_7}$. Their two common endpoints already use all four directions. The remaining endpoint of $A$ lies on the horizontal segment of $P$, where no other ordinary path can share it. The remaining endpoint of $C$ lies on the horizontal segment of $Q$. All endpoints on $P$ lie below that segment in \eqref{eq7}, and above it in \eqref{eq8}; thus no other ordinary path can share that endpoint either. Since ordinary paths can meet only through common host endpoints, no additional ordinary path touches the triple.

\par In \eqref{eq7}, $Z$, if present, leaves $q$ downwards and lies in $y<0$ except at $q$, whereas the three ordinary paths lie in $y\geq0$. In \eqref{eq8}, $Z$ leaves upwards and the three ordinary paths lie in $y\leq0$. None of them contains $q$, so $Z$ does not touch the triple.

\par It remains to consider a path $T$ containing $r=(d,b)$. It must also have an endpoint on $P$. This endpoint cannot lie in the horizontal interior of $P$: a 1-bend path from such a point to $r$ would reach $r$ along $y=b$ and overlap the horizontal segment of $Q$. Since $d>0$ and $b\notin[0,h]$, the only possible form is
\[
T=(0,s)\to(d,s)\to(d,\beta),\qquad 0\leq s<h,
\]
where the final segment contains $r$. In \eqref{eq7}, $s=t$ causes an overlap with $C$ and $s>t$ causes an interior crossing of its vertical segment. Hence $s<t$, in which case its horizontal segment lies below the portions of the triple in $x\geq0$, and its vertical segment lies on $x=d>\alpha$; neither segment meets the triple. In \eqref{eq8}, $s=0$ would overlap the horizontal segment of $C$, so $s>0$. Its horizontal part is then above the triple, and its vertical part is to the right. Again there is no contact.

\par Thus the triple is an isolated connected component. This contradicts the fact that the seven paths induce the connected graph $P_7$. $\square$

\begin{claim}\label{claim3}
When $c<0$, $Z$ can be in contact with at most one ordinary path.
\end{claim}
\par Any contact between $Z$ and an ordinary path must use the endpoint of $Z$ other than $q$. An ordinary endpoint lies in a host interior and therefore cannot be an interior point of $Z$; no ordinary path contains $q$. The remaining endpoint of $Z$ cannot meet two ordinary paths, since they and $Z$ would form a triangle in $P_7$. Hence $Z$ has at most one ordinary neighbor.

\par If $b>0$, every ordinary path lies in $y\geq0$, whereas $Z\setminus\{q\}$ lies in $y<0$. There is no contact. If $b<0$, every pair of adjacent ordinary paths lies entirely in $y\leq0$, as we now verify. Their common endpoint cannot be in the horizontal interior of $P$, nor in its vertical interior, since the earlier analysis of the latter possibility forces $b>h$. It cannot be in the horizontal interior of $Q$, which is below every endpoint on $P$. At the bend of $Q$, one of the unused directions points downwards, so a pair of ordinary paths cannot meet there either. The common endpoint must therefore be the bend of $P$ or an interior point $(c,t)$ of the vertical segment of $Q$, with $b<t<0$.

\par In the first case, both ordinary paths have one endpoint at height zero and their other endpoint below it, so both lie in $y\leq0$. In the second case, both other endpoints must be on the horizontal segment of $P$ or at its bend. An endpoint on the vertical interior of $P$ would have to reach $(c,t)$ by a horizontal segment followed by a vertical segment on $x=c$, which would overlap $Q$ between heights $t$ and zero. Thus again all endpoints, and hence both entire paths, have non-positive height. Since $Z\setminus\{q\}$ lies in $y>0$, it cannot touch a vertex of an ordinary edge. It can only join one ordinary isolated vertex. $\square$

\begin{claim}\label{claim4}
 When $c=0$, the contact graph of the ordinary paths consists only of isolated vertices and independent edges, and $Z$ can be in contact with at most one ordinary path.
\end{claim}
\par Since $c=0$, $q$ is the bend of $P$ and $b<0$. In this case ordinary paths are pairwise disjoint. On the horizontal segment of $P$, all ordinary paths must leave downwards, so a pair cannot share an endpoint. On its vertical segment, two paths leaving to opposite sides could reach $Q$ only if its horizontal segment extended to both sides of $x=0$, which it does not. On the horizontal segment of $Q$, they can approach only from above. On its vertical segment, they can approach only from the left: a path starting on the vertical arm of $P$ would need two bends or would overlap that arm or $Q$, whereas a path starting on the horizontal arm reaches it from the left. At the bend of $Q$, the downward unused direction cannot lead to an endpoint on $P$. The remaining host points are $q,r$, which ordinary paths avoid. These observations exclude every possible common host endpoint. Adding $Z$ can therefore create at most one edge. $\square$

\par After deleting the vertices of $B$, Claims \ref{claim1}-\ref{claim4} show that the contact graph of ordinary paths consists only of isolated vertices and independent edges, the induced subgraph of $P_7$ thus has components of order at most two. By \eqref{eq6}, $|B|\leq2$. Deleting at most two adjacent vertices from $P_7$ leaves at most two components. Hence $|V(P_7-B)|\leq2\cdot2 = 4$, and 
\[
7=|B|+|V(P_7-B)|\leq |B|+2+2\leq6,
\]
a contradiction. This proves that $P_v$ has at most one free endpoint, and by symmetry the same holds for $P_w$.
\end{proof}

\par Fig. \ref{fig3}(b) presents a $B_1$-CPG representation of $H_1$ in which $P_v$ has exactly one free endpoint. Let 
\[
P_v=(-2,0)\to(3,0)\to(3,3); P_w=(5,3)\to(2,3)\to(2,1); P_{u_1}=(3,2)\to(4,2)\to(4,3);
\]
\[
P_{u_2}=(3,2)\to(2,2); P_{u_3}=(2,2)\to(1,2)\to(1,0); P_{u_4}=(1,1)\to(2,1)\to(2,0);
\]
\[
P_{u_5}=(2,1)\to(3,1); P_{u_6}=(3,1)\to(5,1)\to(5,4); P_{u_7}=(5,4)\to(3,4)\to(3,3).
\]
The five triple-contact points represent the triangles $vwu_7$, $vu_1u_2$, $wu_2u_3$, $wu_4u_5$, and $vu_5u_6$. The remaining contacts represent $wu_1$, $vu_3$, $vu_4$, $wu_6$, $u_3u_4$, and $u_6u_7$. Their locations are, respectively, $(4,3)$, $(1,0)$, $(2,0)$, $(5,3)$, $(1,1)$, and $(5,4)$. The triple-contact points are, in the same triangle order, $(3,3)$, $(3,2)$, $(2,2)$, $(2,1)$, and $(3,1)$. These contacts give all twenty-one edges. There are no other contacts and no intersections between path interiors. The endpoint $(-2,0)$ of $P_v$ is free, and its leftward continuation ray avoids every path.

\begin{figure}
\centering
  \setlength{\unitlength}{1bp}%
   \begin{picture}(346.03, 308.99)(0,0)
  \put(0,0){\includegraphics{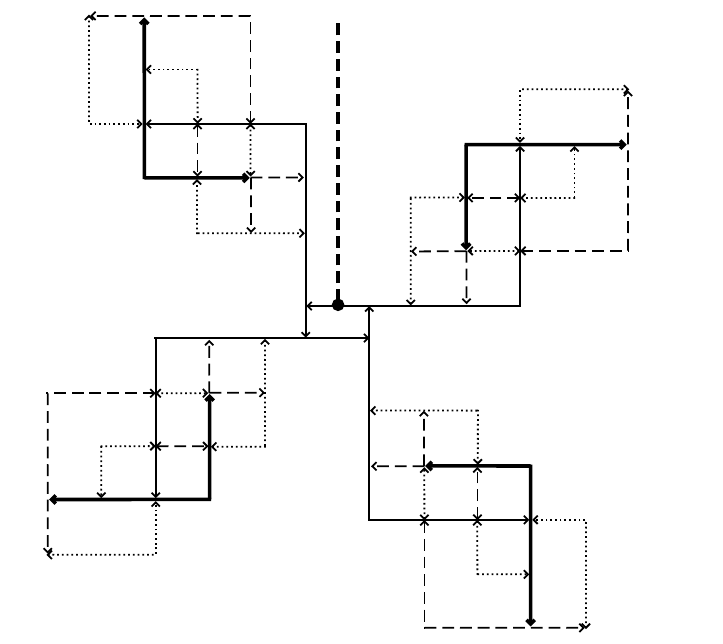}}
  \put(157.67,152.34){\fontsize{9.72}{11.66}\selectfont $p_0$}
  \put(304.63,218.42){\fontsize{9.72}{11.66}\selectfont $H^0_1$}
  \put(283.11,31.66){\fontsize{9.72}{11.66}\selectfont $H^1_1$}
  \put(5.67,84.33){\fontsize{9.72}{11.66}\selectfont $H^2_1$}
  \put(24.52,269.24){\fontsize{9.72}{11.66}\selectfont $H^3_1$}
  \end{picture}%
\caption{A $B_1$-CPG representation of the graph $F$ in which $P_{v_0}$ contains a point $p_0$ from which an upward ray emanates.}
\label{fig4}
\end{figure}

\par For the graph $F$, we choose $v_0$ as its attachment root. This graph is planar; the copies attach at the vertices of a four-cycle.

\begin{lemma}\label{le4.2}
In every $B_1$-CPG representation of $F$, none of the four root paths has a free endpoint. Moreover, $F$ has a $B_1$-CPG representation with an exposed interior point on $P_{v_0}$.
\end{lemma}
\begin{proof}
\par Restrict any representation of $F$ to $H_1^i$. By Lemma \ref{le4.1}, $P_{v_i}$ has at most one endpoint free relative to that copy. A contact representing a newly added edge must use such a relative free endpoint. Otherwise the other root path would also meet a non-root path of the copy, giving an edge not in $F$.

\par The four added edges need four distinct such endpoints. Sharing one endpoint between three added edges would make three roots meet at one point and would add a diagonal of the four-cycle. There are at most four eligible endpoints in total, one at each root. All must therefore be used, while the other endpoint of each root is already occupied inside its copy. This proves the universal endpoint assertion.
\par The $B_1$-CPG representation of $H^0_1$ coincides with that of $H_1$, with the left endpoint of $P_{v_0}$ at $(-2,0)$. Rotating three copies of the $B_1$-CPG representation of $H^0_1$ clockwise by $90^\circ$, $180^\circ$, and $270^\circ$, respectively, so that the top endpoint of $P_{v_1}$ is $(0,0)$, the right endpoint of $P_{v_2}$ is $(-1,0)$, and the bottom endpoint of $P_{v_2}$ is $(-2,-1)$, we obtain a $B_1$-CPG representation of $F$ (Fig. \ref{fig4}). The point $p_0$ has coordinates $(-1,0)$, and no path passes through any point with $y>0$ on the column $x=-1$. Hence $P_{v_0}$ contains a point $p_0$ from which an upward ray emanates.
\end{proof}

\subsection{The reduction}

\par Given an instance $\phi$ of \textsc{Planar Exactly 3-Bounded 3-Sat}, with variable set $\mathcal{V}$ and clause set $\mathcal{C}$, we construct a graph $G(\phi)$ which is a planar $B_1$-CPG graph if and only if $\phi$ is satisfiable. \textsc{Planar Exactly 3-Bounded 3-Sat}: each clause has at most three literals, each variable occurs in exactly three clauses, and the variable-clause incidence graph is planar. Its NP-completeness originates in \cite{dahlhaus1994complexity}. We use the normalized instances and terminal assignment of the reduction in \cite{Champseix2021cpg}. In particular, clauses have two or three literals, and there are enough positive and negative variable terminals to assign distinct terminals to the occurrences.

\par Starting from the planar bipartite graph $I(\phi)$ (the variable–clause incidence graph of $\phi$), with $V(I(\phi))=\mathcal{V}\cup\mathcal{C}$ and $E(I(\phi))=\{xc\mid x\in c \text{ or } x\in\neg c\}$, we replace each variable vertex $x$ with a variable gadget $V_x$ (Fig. \ref{Fig2}(a)), and each clause vertex $c$ with a clause gadget $C_c$ (Fig. \ref{Fig2}(b)). In each subfigure, a double-circle represents the end-eating gadget graph $F$. Moreover, for a vertex $t \in V(G(\phi))$, $t$ is incident to a double-circle if and only if $t$ is adjacent to exactly one vertex $v_0$ of $F$.

\begin{figure}
\centering
  \setlength{\unitlength}{1bp}%
    \begin{picture}(448.20, 134.59)(0,0)
  \put(0,0){\includegraphics{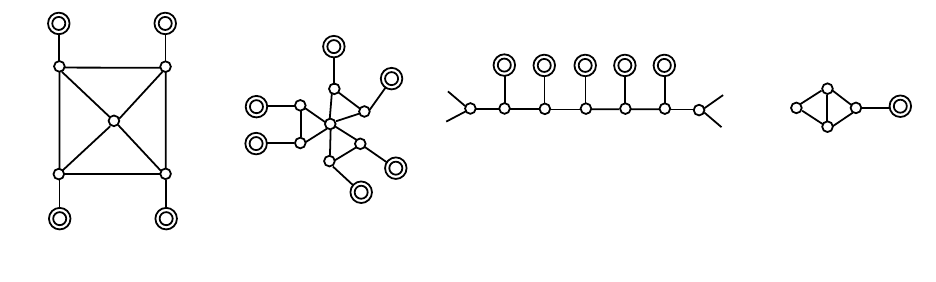}}
  \put(17.07,48.54){\fontsize{9.72}{11.66}\selectfont $d$}
  \put(82.27,48.93){\fontsize{9.72}{11.66}\selectfont $c$}
  \put(17.50,100.26){\rotatebox{0.00}{\fontsize{9.72}{11.66}\selectfont \smash{\makebox[0pt][l]{$a$}}}}
  \put(82.27,100.05){\fontsize{9.72}{11.66}\selectfont $b$}
  \put(51.78,81.43){\fontsize{9.72}{11.66}\selectfont $k$}
  \put(37.70,8.21){\fontsize{9.72}{11.66}\selectfont (a) $V_x$}
  \put(143.86,7.80){\fontsize{9.72}{11.66}\selectfont (b) $C_c$}
  \put(140.36,89.43){\fontsize{9.72}{11.66}\selectfont $p_1$}
  \put(140.42,56.36){\fontsize{9.72}{11.66}\selectfont $q_1$}
  \put(176.50,64.50){\fontsize{9.72}{11.66}\selectfont $q_2$}
  \put(162.51,93.23){\fontsize{9.72}{11.66}\selectfont $q_3$}
  \put(178.20,75.75){\fontsize{9.72}{11.66}\selectfont $p_3$}
  \put(151.43,48.57){\fontsize{9.72}{11.66}\selectfont $p_2$}
  \put(266.78,7.76){\fontsize{9.72}{11.66}\selectfont (c) $L_{x,c}$}
  \put(222.70,72.37){\fontsize{9.72}{11.66}\selectfont $t_{xc}$}
  \put(325.91,72.86){\fontsize{9.72}{11.66}\selectfont $t_{cx}$}
  \put(371.49,80.24){\fontsize{9.72}{11.66}\selectfont $o$}
  \put(360.00,7.76){\fontsize{9.72}{11.66}\selectfont (d)~False~terminator}
  \put(409.56,73.58){\fontsize{9.72}{11.66}\selectfont $t_f$}
  \put(50.08,104.33){\rotatebox{0.00}{\fontsize{9.72}{11.66}\selectfont \smash{\makebox[0pt][l]{$x$}}}}
  \put(50.08,40.96){\rotatebox{0.00}{\fontsize{9.72}{11.66}\selectfont \smash{\makebox[0pt][l]{$x$}}}}
  \put(81.16,76.49){\rotatebox{0.00}{\fontsize{9.72}{11.66}\selectfont \smash{\makebox[0pt][l]{$\neg x$}}}}
  \put(5.67,76.49){\rotatebox{0.00}{\fontsize{9.72}{11.66}\selectfont \smash{\makebox[0pt][l]{$\neg x$}}}}
  \end{picture}%
  \caption{(a) The variable gadget $V_x$. (b) The clause gadget $C_c$. (c) The connector $L_{xc}$. (d) The false terminator $t_f$.}
  \label{Fig2}
\end{figure}

\par For each variable $V_x$, take a wheel with center $k$ and rim $a,b,c,d,a$. Its negative terminals are $(b,c)$ and $(d,a)$, and its positive terminals are $(a,b)$ and $(c,d)$. For each clause $C_c$, take a vertex $o$ joined to the six endpoints of three independent edges $p_iq_i$, $1\leq i\leq3$. The three pairs $(p_i,q_i)$ are its terminals.
\par We replace each edge $xc$ in $I(\phi)$ with a connector $L_{x,c}$ (Fig. \ref{Fig2}(c)). The vertex $t_{xc}$ is linked to a positive or negative terminal of $V_x$, depending on whether $x$ occurs positively or negatively in $c$; and the vertex $t_{cx}$ is linked to a terminal of $C_c$. The underlying path of $L_{xc}$ has vertices $t_{xc},z_1,\dots,z_5,t_{cx}$ in this order. We carry out these transformations so that no terminal is used twice. Finally, if a clause $c$ contains only two literals, we link the last free terminal to a false terminator (Fig. \ref{Fig2}(d)).

\par In Section 4.2 of \cite{Champseix2021cpg}, $G_1(\phi)$ denotes our graph $G(\phi)$ in which every double-circle is the non-planar end-eating gadget $\varepsilon_1$ rather than the planar graph $F$. Consequently, deleting all double-circles from $G_1(\phi)$ and $G(\phi)$; we denote this resulting graph by $D(\phi)$. Fig. \ref{fig5}(b) shows an example $D(\phi)$ for $\phi=c_1\wedge c_2\wedge c_3=(x\vee y)\wedge(x \vee\neg y)\wedge(\neg x\vee y)$, where the black vertices are either variable terminals or clause terminals. Moreover, in \cite[Lemma~8]{Champseix2021cpg}, the vertices of the gadget $\varepsilon_1$ that are adjacent to vertices of $V(D(\phi))$ have $B_1$-CPG representations with no free endpoint. This property coincides with the property that, in every $B_1$-CPG representation of $F$, the path $P_{v_0}$ has no free endpoint, by Lemma \ref{le4.2}. We may therefore replace all $B_1$-CPG representations of $\varepsilon_1$ in a $B_1$-CPG representation of $G_1(\phi)$ by the $B_1$-CPG representation of $F$, thereby obtaining a $B_1$-CPG representation of $G(\phi)$.
\par Moreover, we observe that the vertices of $D(\phi)$ that are adjacent to $v_0$ of $F$ in $G(\phi)$ are exactly the following: the four rim vertices $a,b,c,d$ in each variable gadget; the vertices $p_i,q_i$ for $1\leq i\leq 3$ in each clause gadget; each false terminator $t_f$; and the vertices $z_1,z_2,\dots,z_5$ in each connector $L_{x,c}$. We denote by $S(\phi)$ the set of these vertices.

\par Figures 15 and 16 in \cite{Champseix2021cpg} provide $B_1$-CPG representations of $V_x$, in which the paths representing $a,b,c,d$ each have a free endpoint; Figure 17 provides a $B_1$-CPG representation of $C_c$, in which the paths representing $p_i$ and $q_i$ $(1\leq i\leq 3)$ each have a free endpoint.

\begin{figure}
\centering
  \setlength{\unitlength}{1bp}%
  \begin{picture}(466.00, 244.28)(0,0)
  \put(0,0){\includegraphics{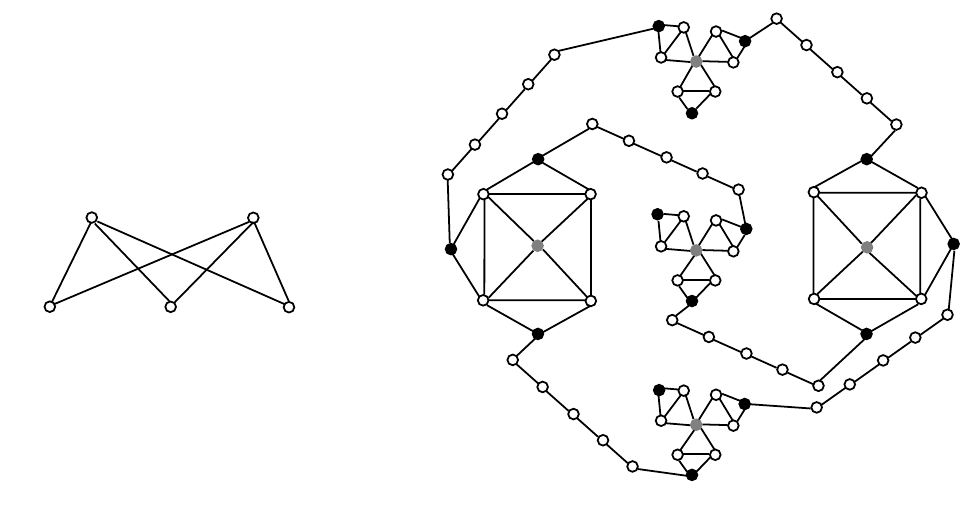}}
  \put(5.67,87.03){\fontsize{9.72}{11.66}\selectfont $c_1=x\vee y$}
  \put(58.50,87.03){\fontsize{9.72}{11.66}\selectfont $c_2=x\vee \neg y$}
  \put(38.65,146.39){\fontsize{9.72}{11.66}\selectfont $x$}
  \put(116.20,146.00){\fontsize{9.72}{11.66}\selectfont $y$}
  \put(120.96,86.40){\fontsize{9.72}{11.66}\selectfont $c_3=\neg x\vee y$}
  \put(58.91,56.73){\fontsize{9.72}{11.66}\selectfont (a) $I(\phi)$}
  \put(341.36,7.76){\fontsize{9.72}{11.66}\selectfont (b) $D(\phi)$}
  \end{picture}%
\caption{(a) The variable-clause incidence graph of $\phi=c_1\wedge c_2\wedge c_3$. (b) The graph $D(\phi)$. When every white vertex and every degree-two black vertex is adjacent to the vertex $v_0$ of a copy of $F$, the graph $G(\phi)$ is obtained. In the variable gadget, the upper and lower variable terminals represent the positive literals of $x$ and $y$, while the remaining variable terminals represent the negative literals of $x$ and $y$.}
\label{fig5}
\end{figure}

\par For completeness, we recall the statement of \cite[Lemma 9]{Champseix2021cpg},  in the form needed here: given an instance $\phi$ of \textsc{Planar Exactly 3-Bounded 3-Sat}, the graph $G_1(\phi)$ obtained from $D(\phi)$ by attaching a copy of the non-planar end-eating gadget $\varepsilon_1$ at every vertex of $S(\phi)$ is $B_1$-CPG if and only if $\phi$ is satisfiable. We obtain the following lemma.

\begin{lemma}\label{le4.3}
$\phi$ is satisfiable if and only if $D(\phi)$ admits a $B_1$-CPG representation in which $P_t$ has a free endpoint for every $t\in S(\phi)$.
\end{lemma}

\par From the properties of $B_1$-CPG representations of $F$, we obtain the following result.

\begin{theorem}\label{th4.1}
Recognizing planar $B_1$-CPG graphs is NP-complete even for graphs of maximum degree at most 11.
\end{theorem}
\begin{proof}
\par Given a formula $\phi$, we construct the planar graph $G(\phi)$ in polynomial time. If $\phi$ is satisfiable, Lemma \ref{le4.3} gives a $B_1$-CPG representation of $D(\phi)$ with all required free endpoints. Lemmas \ref{le2.2} and \ref{le4.2} attach the copies of $F$ without increasing the bend count. Conversely, deleting the copies from any $B_1$-CPG representation of $G(\phi)$ yields a $B_1$-CPG representation of $D(\phi)$ in which $P_t$ has a free endpoint for every $t\in S(\phi)$. Lemma \ref{le4.3} then makes $\phi$ satisfiable.

\par Each universal vertex of $H_1$ has degree eight. Adding the root four-cycle makes the four roots have degree ten, and the attachment root has degree eleven after the edge to the core is added (an example of $G(\phi)$ is obtained from Fig. \ref{fig5}(b) by joining every white vertex and every degree two black vertex to the vertex $v_0$ of a copy of $F$). Other vertices of a copy have degree at most eight. Thus 
\[
\Delta(G(\phi))=11.
\]
Membership in NP follows from Lemma \ref{le2.1}.
\end{proof}

\section{Conclusion}
\par In this paper, we resolve the complexity of recognizing planar $B_k$-CPG graphs for $k\in\{0,1\}$, a problem posed by Champseix et al. \cite{Champseix2021cpg}. We achieve the recognition of planar $B_0$-CPG graphs by a reduction from the HV-Rectilinear Planarity problem for planar subcubic graphs. For planar $B_1$-CPG graphs, we start from the graphs constructed in \cite{Champseix2021cpg} for recognizing $B_1$-CPG graphs, modify the end-eating gadget, and thereby obtain a planar $B_1$-CPG graph.

\par The case $k=2$ remains open. The reduction of Section 4 would extend directly to two bends if a suitable endpoint gadget were available. Specifically, suppose that there exists a planar graph $F_2$ such that, in every $B_2$-CPG representation of $F_2$, some vertex $v$ has a path $P_v$ with no free endpoint and contains an exposed interior point. Then replacing the end-eating gadget $F$ in Section 4 by $F_2$ would reduce \textsc{Planar Exactly 3-Bounded 3-Sat} to recognizing planar $B_2$-CPG graphs, exactly as in the one-bend case. The remaining obstacle is therefore the construction of such a graph $F_2$.

\vspace{6mm}
\noindent\textbf{Acknowledgements} This work was funded in part by the Scientific Research Projects of the Department of Education of Anhui Province (Grants No. 2025AHGXZK40333).
    
\vspace{2mm}
\noindent\textbf{Data Availability} Data sharing is not applicable to this article as no datasets were generated or analyzed during the current study.
	
\vspace{2mm}
	
\noindent\textbf{Conflict of interest} The authors have no relevant financial or non-financial interests to disclose.

\end{document}